\documentclass[twoside]{article}

\usepackage[accepted]{aistats2020-anno}
\usepackage{maths}

\usepackage[round]{natbib}

\hypersetup{
	colorlinks,
	linkcolor={red!40!gray},
	citecolor={blue!40!gray},
	urlcolor={blue!70!gray}
}

\begin{document}

%

%
\renewcommand*{\thefootnote}{\fnsymbol{footnote}}
\twocolumn[

\aistatstitle{Convergence Analysis of Block Coordinate Algorithms with Determinantal Sampling}
\aistatsauthor{ Mojm\'ir  Mutn\'y\footnote[2]{Equal contribution.}   \And Micha\l{}  Derezi\'nski\footnotemark[2]  \And Andreas Krause  }
\aistatsaddress{  Department of Computer Science\\
	ETH Zurich, Switzerland \\ \texttt{mmutny@inf.ethz.ch} \And Department of Statistics \\ University of California, Berkeley \\ \texttt{mderezin@berkeley.edu} \And Department of Computer Science\\ 
	ETH Zurich, Switzerland \\ \texttt{krausea@inf.ethz.ch} } 
]

\begin{abstract}
	We analyze the convergence rate of the randomized Newton-like method
        introduced by \cite{Qu2015Feb} for smooth and convex
        objectives, which uses random coordinate
        blocks of a Hessian-over-approximation matrix
        $\bM$ instead of the true Hessian. The convergence analysis of the algorithm is  
        challenging because of its complex dependence on the structure
        of $\bM$. However, we show that when the coordinate blocks are
        sampled with probability 
        proportional to their determinant, the convergence rate depends solely on the eigenvalue distribution of matrix $\bM$, and has an analytically tractable form. To 
        do so, we derive a fundamental new expectation formula for
        determinantal point processes. We show that determinantal
        sampling allows us to reason about the optimal subset size of
        blocks in terms of the spectrum of $\bM$. Additionally, we
        provide a        numerical evaluation of our analysis,
        demonstrating cases where determinantal sampling is superior
        or on par with uniform sampling.
\end{abstract}

	\section{INTRODUCTION}

We study unconstrained optimization of the form:
	\[	\min_{x \in \mR^d} f(x), 	\]
	where we assume that the function $f:\mR^d \rightarrow \mR$ is smooth, convex, and potentially high dimensional. This problem commonly arises in empirical risk minimization \citep[ERM, see][]{Shalev-Shwartz2014}.
State-of-the-art approaches for minimization of convex ERM objectives with large numbers of data points include variants of stochastic gradient descent (SGD) such as SVRG \citep{Johnson2013}, SARAH \citep{Nguyen2017} and a plethora of others. Alternatively, one can approach the ERM problem via a dual formulation, where fast coordinate minimization techniques such as SDCA \citep{Shalev-Shwartz2013}, or parallel coordinate descent \citep{PCDM,Richtarik2015a} can be applied. This is especially desirable in distributed and parallel environments \citep{HYDRA,Ma2015,Duenner2016}. These approaches are closely related to methods that subsample the Hessian \citep{Pilanci2015, subsampled-newton, Roosta-Khorasani2016}.
	
We study a block coordinate descent algorithm
first introduced by \cite{Qu2015Feb}. In each iteration of this algorithm, we sample a block of coordinates and then solve a Newton step on the chosen coordinate subspace.
However, in place of the true Hessian, a fixed over-approximation matrix $\bM$ is used for the sake of efficiency. The Newton step is computed on a sparsified version of this matrix with all but the selected coordinates set to zero, denoted $\bM_{\hat{S}}$ (see Section \ref{s:notation} for the complete notation).
Originally, \cite{Qu2015Feb} called this method Stochastic Dual Newton Ascent (SDNA), appealing to the fact that it operates in a dual ERM formulation. Later, it was also called a Stochastic Newton method \citep{Mutny2018a}, while we use the name \emph{Randomized Newton Method} (RNM) following \cite{Gower2019}\footnote{\citet{Gower2019} consider a more general algorithm, relying on the novel assumptions of \emph{relative} smoothness and convexity. We discuss this setting in Appendix \ref{appendix:relative}.}.


	The sampling strategy for the coordinate blocks has a dramatic impact  on  the convergence rate \citep{Qu2016}. \cite{Gower2015} demonstrate that by optimizing the sampling probabilities one can obtain very significant speedups, however this optimization is a semidefinite program which may be even more challenging than the original optimization problem itself. Even when using a basic sampling strategy (such as uniform), the convergence analysis of RNM is challenging because it hinges on deriving the \emph{expected pseudoinverse} of $\bM_{\hat{S}}$, henceforth denoted  $\mE[(\bM_{\hat{S}})^+]$. Prior to this work, no simple closed form expression was  known for this quantity.

To overcome this challenge, we focus on a strategy of randomly sampling  blocks of coordinates proportionally to the determinant of the corresponding submatrix of $\bM$, which we call \emph{determinantal sampling}. Similar sampling schemes have been  analyzed in the context of stochastic optimization before \citep{dpp-minibatch, Borsos2019}. Recently, \citet{Rodomanov2019} analyzed determinantal sampling for randomized block coordinate descent, however they imposed cardinality constraints on the block size, and as a result were unable to obtain a simple expression for $\mE[(\bM_{\hat{S}})^+]$.

We use determinantal sampling with randomized block size, which allows us to obtain a simple closed form expression for the expected pseudoinverse:
  \[ \mE[(\bM_{\hat{S}})^+] = (\alpha\bI + \bM)^{-1}, \]
where $\alpha$ is a tunable parameter that is used to control the expected block size. 
With the use of this new expectation formula, we establish novel bounds on the convergence rate of RNM depending on the spectral properties of the over-approximating matrix $\bM$. For many instances of the problem, the matrix coincides with the data covariance, and spectral decays of such covariances are  well understood \citep{Blanchard2007}. This allows us to predict the decay-specific behavior of RNM with determinantal sampling and recommend the optimal block size. 
		
The cost of each iteration of RNM scales cubically with the size of the block due to matrix inversion. \citet{Qu2015Feb} demonstrate numerically that for small blocks the optimization time decreases but at some point it starts to increase again. They surmise that the improvement is obtained only as long as the inversion cost is dominated by
the other fixed per-iteration costs such as fetching from memory. However, whether the only possible speedup stems from this has remained unclear. We answer this question for determinantal sampling by deriving the optimal subset size in the case of kernel ridge regression. We show that when the eigenvalue decay is sufficiently rapid, then the gain in convergence rate can dominate the cost of inversion even for larger block sizes. 
	
		
\subsection{Contributions}
The main contributions of this paper can be summarized as follows:
	\begin{itemize}
        \item We obtain a novel and remarkably simple expectation formula for determinantal sampling that allows us to derive a simple and closed form expression for the convergence rate of the Randomized Newton Method. 
		
		\item This allows us to improve the previous bounds on the theoretical speedup of using coordinate blocks of larger sizes. For example, we show that in the case of kernel regression with a covariance operator that has exponentially decreasing spectrum, the theoretical speedup is \emph{exponential}. 
		
		\item We take into account the actual per iteration cost, and analyze not only the convergence rate of the algorithm, but also its numerical effort to solve a problem up to some relative precision. This allows us to classify the problems into categories where the optimal block size is one, the full matrix, or somewhere in between.
		
		\item We numerically validate the discovered theoretical properties of \emph{determinantal sampling}, and demonstrate cases when it improves over  uniform sampling, and when it performs similarly. If the two perform similarity, our analysis serves as a more interpretable proxy for the convergence analysis of uniform sampling.
		
	\end{itemize}
	
	\subsection{Notation}
        \label{s:notation}
	
	Let $S$ be a non-empty subset of  $[d]:=\{1,2, \dots, d\}$. We let $\bI_{:S}$ be the $d \times |S|$ matrix  composed of columns $i \in S$ of the $d\times d$ identity matrix $\bI$. Note that $\bI_{:S}^\top \bI_{:S}$ is the $|S|\times |S|$ identity matrix.	
	Given an invertible matrix $\bM\in \mR^{d\times d}$, we can extract its principal  $|S| \times |S|$ sub-matrix  with the corresponding rows and columns indexed by $S$ via \( \mathbf{M}_{SS} \eqdef \bI_{:S}^\top \mathbf{M} \bI_{:S}\),
	and additionally keeping the sub-matrix in its canonical place we can define the following operation,
	\begin{equation}\label{eq:slice}
	\bM_S \eqdef \bI_{:S}\bM_{SS} \bI_{:S}^\top.
	\end{equation}	
	Note that $\mathbf{M}_{S}$ is the $n\times n$ matrix obtained from $\bM$ by retaining elements $\bM_{ij}$ for $i \in S$ and $j \in S$; and all the other elements set to zero. By $(\cdot)^+$ we denote the Moore-Penrose pseudoinverse. The matrix $(\bM_{S})^{+}$ can be calculated by inverting $\bM_{SS} \in \mR^{|S|\times |S|}$, and then placing it back into the $d \times d$ matrix.



\section{ALGORITHM}\label{sec:algo}
        The key assumption that motivates RNM is a smoothness condition
that        goes beyond the standard assumptions in the optimization literature, where smoothness would be characterized by a symmetric quadratic with the radius $L$. Instead, Assumption \ref{ass:smooth} below is tighter, allowing for more refined analysis, and can be related to the standard assumption by $L = \lambda_{\max}(\bM)$.   
	
		\begin{assumption}[Smoothness]\label{ass:smooth}
		There exists a symmetric p.d.~matrix $\mathbf{M} \in \mathbb{R}^{d \times d}$ such that $\forall x,h \in \mathbb{R}^d$,
		\begin{equation}\label{eq:smooth}
			f(x+h)\leq f(x)+\braket{\nabla f(x),h}+\frac{1}{2}\braket{h,\mathbf{M}h}.
		\end{equation}
	\end{assumption}
	
%
	This assumption is satisfied for quadratic problems such as ridge regression with squared loss,
		\(
		y = \bA^\top w + \epsilon,
		\)
		where $\bA \in \mR^{n\times d}$ is the data matrix, and $y$ is the vector of responses, which is corrupted via the noise $\epsilon \in \mR^n$.
	  In this case, Assumption \ref{ass:smooth} holds with $\bM$ being the offset covariance matrix $\bA^\top \bA + \lambda\bI $, where $\lambda$ is the regularization parameter. Beyond quadratic problems, it holds for many common problems such as logistic regression, where $\bM = \frac{1}{4} \bA^\top \bA$. Section~\ref{sec:ridge} provides examples in the dual formulation. 

\subsection{Randomized Newton Method}
		Let $k$ be the iteration count and $x_0$ be the initial point. The Randomized Newton Method algorithm is defined via the following update rule:
		\begin{equation}\label{eq:update}
			x_{k+1} = x_k - \left(\bM_{S_k}\right)^+\nabla f(x_k),
		\end{equation}
		where $S_k \subseteq [d]$ is a subset of coordinates chosen at iteration $k$ from random sampling $\hat{S}$ to be defined. Notice that since $\bM_{S_k}$ is a sparse $d \times d$ matrix with only a $|S_k| \times |S_k|$ principal submatrix that is non-zero, its inversion costs $\mO(|S_k|^3)$ arithmetic operations. Moreover, only  $|S_k|$ elements of $\nabla f(x_k)$ are needed for the update.
Note that if $|S_k|=1$ then we are in the classical case of coordinate descent, while if $S=[d]$, then we are performing a Newton step (with $\bM$ in place of the true Hessian). 
		
%
%
%
	\subsection{Sampling}
        The strategy with which one chooses blocks $S_k \subseteq [d]$ in \eqref{eq:update} is of great importance and it influences the algorithm significantly.
This strategy, called a \emph{sampling} and denoted $\hat{S}$, is a random set-valued mapping with values being subsets of $[d]$. A \emph{proper sampling} is such that $p_i \eqdef  \pP(i \in \hat{S}) > 0$ for all $i$.
		
The most popular are \emph{uniform} samplings, i.e., those for which the marginal probabilities are equal:
		\[ \pP(i \in \hat{S}) = \pP(j \in \hat{S}) \quad \forall i,j \in [d]. \]
This class includes $\tau$-nice and $\tau$-list samplings \citep{Qu2016}. The $\tau$-nice sampling considers all elements of a power set of $[d]$ with a fixed cardinality s.t.~$|\hat{S}| = \tau$. There are $\binom{d}{\tau}$ of such subsets and  each of them is equally probable. Consequently, the probability $\pP(i \in \hat{S}) = \frac{\tau}{n}$. On the other hand, the $\tau$-list sampling is restricted to ordered and consecutive subsets of the power set, with cardinality fixed to $\tau$. 
		
Data dependent (and potentially non-uniform) samplings, which sample according to the diagonal  elements of $\bM$, have been analyzed in the context of coordinate descent \citep{Qu2016,Allen-Zhu2016,Hanzely2018,Richtarik2015a}.

\section{DETERMINANTAL SAMPLING}\label{sec:dpp}
Our proposed sampling for the Randomized Newton Method is based
on a class of distributions called {\em Determinantal Point Processes
(DPPs)}. Originally proposed by \cite{dpp-physics}, DPPs have found numerous
applications in machine learning \citep{dpp-ml} as well as optimization
\citep{dpp-minibatch,Borsos2019}, for their variance reduction properties and the ability to produce diverse samples.
\begin{definition}
  For a $d\times d$ p.s.d.~matrix $\bM$, we define $\DPP(\bM)$ as a distribution over all subsets $S\subseteq [d]$, so
  that 
  \begin{equation}\label{eq:sample_def}
  \pP(S)\propto \det(\bM_{SS}).
  \end{equation}
\end{definition}
Even though this is a combinatorial distribution, the
normalization constant can be computed exactly. We state this well
known fact (e.g., see \cite{dpp-ml}) separately because it is crucial
for proving our main result.
\begin{lemma}[Normalization]\label{l:normalization}
  For a $d\times d$ matrix $\bM$,
  \begin{align*}
    \sum_{S\subseteq [d]}\det(\bM_{SS}) = \det(\bI+\bM).
  \end{align*}
\end{lemma}
Note that the distribution samples out of a power set of $[d]$.
While cardinality constrained versions have also been used,
they lack certain properties such as a simple normalization
constant. Even though the subset size of $\DPP(\bM)$ is a random
variable, it is highly concentrated around its mean, and it can also
be easily adjusted by replacing the matrix with a rescaled version
$\frac1\alpha\bM$, where $\alpha>0$. This only affects the
distribution of the subset sizes, with the expected size given by the following
lemma \citep[see][]{dpp-ml}.
\begin{lemma}[Subset Size]\label{prop:size}
If $\hat S\sim \DPP\left(\frac{1}{\alpha}\bM\right)$, then
  \begin{equation} \label{eq:expected_size}
    \mE[|\hat{S}|] = \Tr(\bM(\alpha \bI + \bM)^{-1}).
  \end{equation}
\end{lemma}
By varying the value of $\alpha$, we can obtain any desired expected
subset size between $0$ and $d$. As we increase $\alpha$, the subset
size decreases, whereas if we take $\alpha\rightarrow 0$, then in the
limit the subset size becomes $d$, i.e., always selecting the $[d]$.
While the relationship between $\alpha$ and $\mE[|\hat{S}|]$ cannot
be easily inverted analytically, it still provides a convenient way of
smoothly interpolating between the full Newton and coordinate
descent. To give a sense of what $\alpha$ can be used to ensure
subset size bounded by some $k$, we give the following lemma.
\begin{lemma}\label{lemma:lambdas_exp}
  Let $\{\lambda_i\}_{i=1}^d$ be the eigenvalues of $\bM$ in a decreasing order. 
If $\alpha = \sum_{j\geq k} \lambda_j$, then $\mE[|\hat{S}|] < k$.
\end{lemma}
\subsection{New expectation formula}
We are now ready to state our main result regarding DPPs, which is a new
expectation formula that can be viewed as a matrix
counterpart of the determinantal identity from Lemma
\ref{l:normalization}. 
\begin{theorem}\label{t:main}
	If $\bM\succ \mathbf{0}$ and $\hat S\sim \DPP\left(\frac{1}{\alpha}\bM\right)$, then
	\begin{equation}
	\mE\big[ (\bM_{\hat S})^+\big] = (\alpha\bI+\bM)^{-1}.\label{eq:expectation}
	\end{equation}
\end{theorem}
\begin{remark}
If we let $\bM\succeq \mathbf{0}$, then the equality
in \eqref{eq:expectation} must be replaced by a p.s.d.~inequality $\preceq$.
\end{remark}
We postpone the proof to the appendix. The remarkable simplicity of our result
leads us to believe that it is of interest not only in the context of the Randomized Newton Method, but also to the broader DPP community. While some
matrix expectation formulas involving the pseudoinverse have been
recently shown for some special DPPs
\citep[e.g.,][]{unbiased-estimates-journal},
this result for the first time relates an \emph{unregularized} subsampled
pseudoinverse with a $\alpha\bI$-\emph{regularized} inverse of the full matrix
$\bM$. Moreover, the amount of  regularization that appears in the
formula is directly related to the expected sample size. 

\subsection{Efficient sampling}
Efficient DPP sampling has been an active area of research over the
past decade. Several different approaches have been developed,
such as an algorithm based on the eigendecomposition of $\bM$
\citep{dpp-independence,dpp-ml} as well as an approximate MCMC sampler
\citep{rayleigh-mcmc} among others. For our problem, it is important
to be able to sample 
from $\DPP(\bM)$ without having to actually construct the entire
matrix $\bM$, and much faster than it takes to compute the full
inverse $\bM^{-1}$. Moreover, being able to rapidly generate multiple independent
samples is crucial because of the iterative nature of the Randomized
Newton Method. A recently proposed DPP sampler satisfies all of these
conditions. We quote the time complexity of this method (the bounds
hold with high probability relative to the randomness of the algorithm).
\begin{lemma}[\cite{dpp-sublinear}]\label{t:dpp-algo}
For a $d\times d$ p.s.d.~matrix $\bM$ let $k=\mE[|\hat{S}|]$ where
$\hat{S} \sim\DPP(\bM)$. Given $\bM$, we can sample
\begin{enumerate}
\item the first $\hat S$ in:\quad
  $d\cdot\mathrm{poly}(k)\,\mathrm{polylog}(d)$ time, 
\item each next sample of $\hat S$ in:\hspace{6mm}  $\mathrm{poly}(k)$ time.
  \end{enumerate}
\end{lemma}
Note that the time it takes to obtain the first sample (i.e., the
preprocessing cost) is $o(d^2)$,
meaning that we do not actually have to read the entire matrix
$\bM$. Moreover, the cost of producing repeated samples only depends
on the sample size $k$, which is typically small. The key idea behind
the algorithm of \cite{dpp-sublinear} is to produce a
larger sample of indices drawn i.i.d. proportionally to the marginal probabilities
of $\DPP(\bM)$. For any $i\in [d]$, the marginal probability of $i$ in
$\hat{S} \sim \DPP(\frac1\alpha\bM)$ is:
\begin{align*}
  \pP(i\in \hat{S}) = \big[\bM(\alpha\bI+\bM)^{-1}\big]_{ii}.
\end{align*}
In the randomized linear algebra literature, this quantity is often
called the $i$th $\alpha$-ridge leverage score
\citep{ridge-leverage-scores}, and sampling i.i.d.~according to ridge
leverage scores is known to have strong 
guarantees in approximating p.s.d.~matrices.

Approximate ridge leverage score sampling incurs a
smaller preprocessing cost compared to a DPP \citep{Calandriello2017},
and basically no resampling cost. Motivated 
by this, we propose to use this sampling as  
a fast approximation to $\DPP(\frac1\alpha\bM)$ and our experiments
demonstrate that it exhibits similar convergence properties for
Randomized Newton.
We numerically compare the sampler from Lemma~\ref{t:dpp-algo}
against leverage score sampling in Appendix \ref{a:leverage}.

\section{CONVERGENCE ANALYSIS}\label{sec:analysis}

	In this section, we analyze the convergence properties of the
        update scheme \eqref{eq:update} with determinantal sampling
        defined by \eqref{eq:sample_def}. In order to establish linear
        rate of convergence, we need to assume strong convexity.
		\begin{assumption}[Strong Convexity]\label{ass:strongconvex}
		Under Assumption \ref{ass:smooth}, there exists a $\kappa > 0$ such that $\forall x,h \in \mathbb{R}^d$,
		\begin{equation*}\label{eq:strgcnvx}
		f(x)+\braket{\nabla f(x),h}+\frac{\kappa}{2}\braket{h,\mathbf{M}h}\leq f(x+h)
		\end{equation*}
	\end{assumption}
	
	 Intuitively, the parameter $\kappa \in (0,1]$ measures the
         degree of accuracy of our quadratic approximation. For a
         quadratic function $\kappa = 1$. 
	
	\begin{lemma}[\cite{Qu2015Feb}]\label{theorem:1}
          Under Assumptions \ref{ass:smooth} and \ref{ass:strongconvex},
let $\{x^k\}_{k\geq0}$ be a sequence of random vectors produced by the
Algorithm with a proper sampling $\hat{S}$, and let $x^*$ be the optimum
of $f$. Then, 
		\begin{equation*}
		\mE\big[f(x^{k+1})-f(x^*)\big]\leq
                \big(1-\sigma(\hat{S})\big)\mE\big[f(x^k)-f(x^*)\big] ,
		\end{equation*}
		where
		\begin{equation}\label{eq:sigma_1}
		\sigma (\hat{S}) \eqdef \kappa\cdot
                \lambda_{\min}\big(\mathbf{M}^{1/2}
                \mathbb{E}[(\mathbf{M}_{\hat{S}})^{+}]\mathbf{M}^{1/2}\big). 
		\end{equation}
	\end{lemma}
Strong convexity is not necessary to run RNM (\ref{eq:update}). In the
cases where the function is only convex, we recover the standard sublinear
rate depending on $\sigma$. 
\begin{lemma}[\cite{Karimireddy2018a}]\label{thm:convex}
Let $f$ be convex and satisfy Assumption \ref{ass:smooth}. Then using
the update scheme in \eqref{eq:update} with any proper 
sampling,
\[\mE[f(x^k) - f(x^*)] \leq \frac{2D}{\sigma(\hat{S}) k}	\] 
where $\sigma(\hat{S})$ is as in \eqref{eq:sigma_1}, and $D = \max_{x}
\{ (x^* -x)^\top \bM (x^* -x) | f(x) \leq f(x^0) \}$ is the set
diameter in $\bM$ geometry at the initial level sets.   
\end{lemma}

The preceding two lemmas introduced the quantity $\sigma(\hat{S})$ characterizing
the theoretical convergence rate of the method. By
applying our new expectation formula (Theorem \ref{t:main}) we obtain
a simple form for this
quantity under DPP sampling.

		\begin{theorem}\label{corr:recurence}
Under Assumption \ref{ass:smooth}, given $\alpha>0$:
		\begin{equation}\label{eq:corr}
\sigma(\hat{S})  = \kappa \frac{ \lambda_d }{ \lambda_d +
  \alpha}\quad\text{for}\quad\hat{S} \sim \DPP\big(\tfrac{1}{\alpha} \bM \big),
		\end{equation}
		where $\lambda_d = \lambda_{\min}(\bM)$. 
	\end{theorem}
	
Note that $\sigma(\hat{S})$ depends solely on the
smallest eigenvalue and the parameter $\alpha$ controlling the
expected size. This is not the case for other samplings, and other
closed forms are not known in general \citep{Qu2015Feb}.  

Recall that the smaller the $\alpha$ the bigger the subsets. The
closed form expression from Theorem \ref{corr:recurence} combined with
Lemma~\ref{lemma:lambdas_exp} allows us to formulate a  
recurrence relation between the convergence rates with different
expected set sizes.  
		\begin{proposition}[Recurrence relation]\label{prop:recurence}
  Let $\{\lambda_i\}_{i=1}^d$ be the eigenvalues of $\bM$ in a
  decreasing order. Let $k<d$ be a positive integer, $\alpha(k) = \sum_{i > k-1}\lambda_i$, and $\sigma(k) = \frac{\lambda_d}{\lambda_d + \alpha(k)}$. Then, 
		\[  \sigma(k) = \frac{\sigma(k+1)}{1 + \frac{\lambda_{k}}{\lambda_d}\sigma(k+1)}  \] 
		and $\sigma(d) = 1$ while $\mE[|S|] < k$. 
	\end{proposition}
	
	This result allows us to further improve the theoretical
        bounds from \cite{Qu2015Feb} on the parameter
        $\sigma$. Namely, it has been previously established that
        $\sigma$ grows at least linearly with the increasing subset
        size of $\tau$-uniform sampling, i.e., $\tau\sigma(1) \leq
        \sigma(\tau)$. We can establish more informative bounds
        depending on the eigenvalue decay. Specifically, for a
        decreasing sequence of eigenvalues $\{\lambda_i
        \}_{i=1}^d$,   	
	\begin{equation}\label{eq:tighter}
		\left(1 + \sum_{j=1}^{\tau-1}\frac{\lambda_j}{\lambda_d}\right)\sigma(1) \leq \sigma(\tau).
	\end{equation}

	For example, given exponentially decaying eigenvalues
        $\lambda_i = \gamma^i$ where $\gamma < 1$, the increase is at
        least exponential, and the convergence rate is at least $(1+(\tau-1)\gamma^{\tau-d})$ bigger.  The case with linear speed-up is recovered when all eigenvalues are equal.


\section{OPTIMAL BLOCK SIZE}\label{sec:decay}

Our results such as Proposition \ref{prop:recurence} and inequality
\eqref{eq:tighter} describe the convergence speedup of using larger
coordinate blocks for RNM with determinantal sampling as a function of
the eigenvalues of $\bM$. In this 
section, we demonstrate that covariance matrices  arising in kernel
ridge regression have known asymptotic eigenvalue decays, which allows
for a precise characterization of RNM performance. 

\subsection{Kernel Ridge Regression} \label{sec:ridge}
The motivating example for our analysis is the dual formulation of
kernel ridge regression which is a natural application for block
coordinate descent because of its high dimensionality.
Suppose our (primal) regression problem is defined by the following
objective:
\[
\min_{x \in \mR^d}	\frac{1}{n}  \sum_{i=1}^{n}  \frac{1}{2}(\Phi(a_i)^\top x - y_i)^2 + \frac{\lambda}{2} \norm{x}^2_2 ,
\]
where $\Phi(\cdot)$ represents the kernel feature mapping and $\lambda$ is the
regularization parameter. Due to the Fenchel duality theorem \citep{Borwein2005},
the dual formulation of this problem is:
\begin{equation}\label{eq:dual}
\min_{\alpha \in \mR^n} \frac{1}{2n} \alpha^\top \bK \alpha + \frac{\lambda}{2}\sum_{i=1}^{n}\left(\alpha_i^2 + 2\alpha_i y_i\right),
\end{equation}
where $\bK_{ij} = \Phi(a_i)^\top\Phi(a_j)$. It is easy to see that the
minimization problem \eqref{eq:dual} is exactly in the right form
for RNM to be applied with the matrix $\bM = \frac{1}{n}\bK + \lambda
\bI$. Notice that $\bM$ is an $n\times n$ matrix and sampling
sub-matrices of $\bM$ has the interpretation of subsampling the
dataset. However, to keep the notation consistent with earlier
discussion, w.l.o.g.~we will let $d=n$ for the remainder of this
section so that $\bM$ is $d\times d$. We will also assume that the
minimization problem is solved with the RNM update where each
coordinate block is sampled as $\hat{S}\sim\DPP(\frac1\alpha\bM)$.
		\begin{figure*}
			\centering
			\begin{subfigure}[b]{0.48\textwidth}
				\includegraphics[width =
                                1\textwidth]{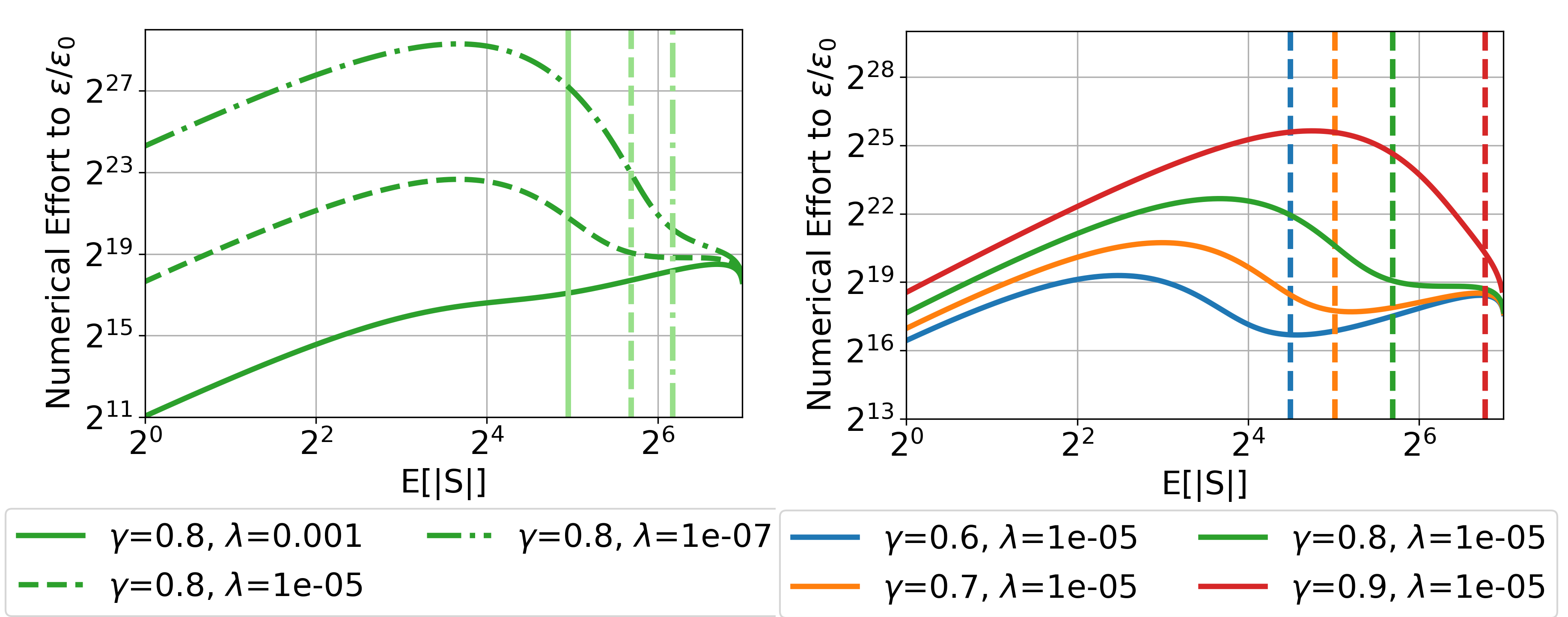}
                                \vspace{2mm}
				\caption{The left column varies
                                  $\lambda$ and the right one varies
                                  $\gamma$. The vertical line
                                  corresponds to the value of $q =
                                  \frac{\log(\lambda)}{\log(\gamma)}$.} 
				\label{fig:gamma}
                              \end{subfigure}
                              \hspace{2mm}
			\begin{subfigure}[b]{0.48\textwidth}
				\includegraphics[width = 1\textwidth]{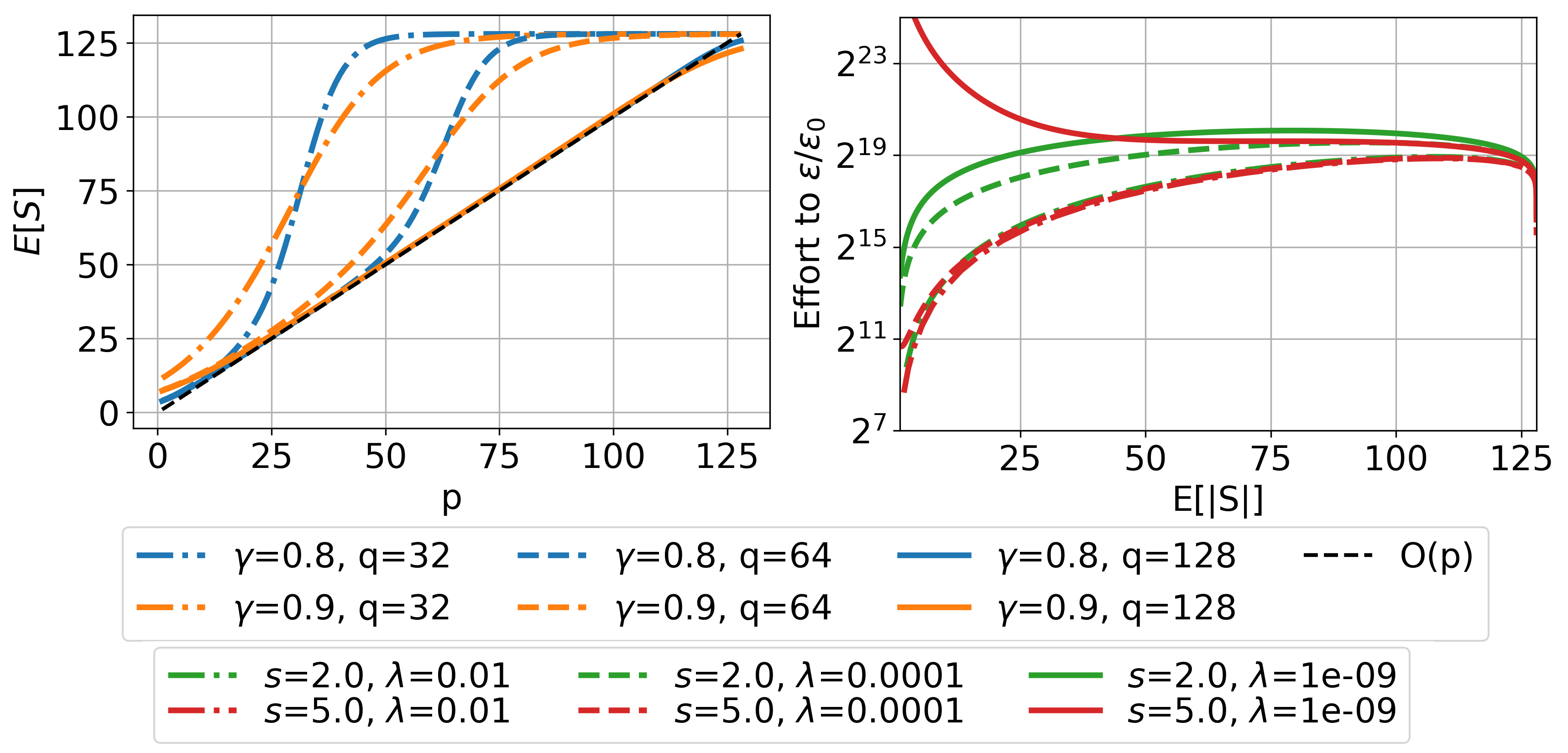}
				\caption{(left) $\mE[|S|]$ vs
                                  parameter p for exponential kernel.\\
                                  ~\quad(right) Numerical effort for polynomial decay.}
				\label{fig:esize}
                              \end{subfigure}
                              \vspace{1mm}
\caption{In (a) we consider exponentially decreasing eigenvalues.
In (b) (left) we plot the relationship
  between $p$ and $\mE[\hat{S}]$ for exponential decay. In (b) (right)
  we show the numerical effort for polynomial decay.}  
			\label{fig:exponential}
		\end{figure*}
	\subsection{Exponentially decreasing spectrum}\label{sec:exponential_decay}
	Let $\{\lambda_i\}_{i=1}^d$ be the eigenvalues of $\bM$ in decreasing order. Suppose that the eigenvalue decay is exponentially decreasing:
        \[ \lambda_i = C \gamma^i + C\lambda  \text{ for } \gamma<1. \]    
        \paragraph{Motivation }
        A classical motivating example for the exponential eigenvalue decay
        is the \emph{squared exponential kernel}, where an
        analytical form of the decay can be derived for normally
        distributed data \citep{Rasmussen2006a}. In particular,
        assuming $x \sim \mN(0,\eta^2)$, and using the kernel
        function $k(x,y) = 
        \exp\big(-\frac{(x-y)^2}{2l^2}\big)$ in one dimension, the 
        eigenvalues satisfy $\lambda_k \leq 
        C\gamma^k$ for a general constant $C$ independent of $k$, 
        where  
   \begin{equation}\label{eq:se_decay}
   \gamma = \frac{2\eta^2}{l^2 + 2\eta^2 + \sqrt{l^2 + 2\eta^2}}.
   \end{equation}
    \paragraph{Complexity} For the ease of exposition, suppose that
    $\gamma^{q+1} \leq \lambda \leq \gamma^q$, where $\lambda$ is
    the regularization constant and $q \in [d]$. Intuitively, this means that the
    regularization parameter flattens the decay at $\gamma^{q+1}$, which will play a role in the analysis. 

    To control the expected size $\mE[|\hat{S}|]$ of determinantal
    sampling, let $\alpha(p) = C\gamma^p$, where $p \in [1,d]$.
We get:
    \[ \mE[|\hat S|] \stackrel{\eqref{eq:expected_size}}  \leq
        p + R_d(p,q)\quad\text{for}\quad\hat S\sim\DPP\big(\tfrac1{\alpha(p)}\bM\big), \] where  $R_d(p,q) = \sum_{i=1}^{d-p}
        \frac{\gamma^i + \gamma^{q-p}}{\gamma^i +
          \gamma^{q-p}+1}$. Asymptotically, if
        $p\ll q$, i.e., the parameter $\alpha(p)$ dominates the
        regularization $C\lambda$, then the expected
        subset is $\mE[|\hat S|] \approx p$. However, in the regime
        where $p=\Omega(q)$, the expected 
        subset size rapidly goes up to $d$ (see Figure \ref{fig:esize}
        (left)).
        
    We  now derive the convergence rate of RNM under determinantal
    sampling:
	\[1 - \sigma(\hat{S}) \stackrel{\eqref{eq:corr}} =
          \frac{C\gamma^p}{C\gamma^d + C\lambda  + C\gamma^p} \geq
          \frac{1}{1 + \gamma^{-p}(\gamma^{d} + \gamma^{q}) }. \] 
Likewise one can see that the convergence rate improves exponentially
with $p$.  
	
	\paragraph{Numerical effort} From Theorem \ref{theorem:1}, we know that in order to reach an $\epsilon$ accurate solution from the initial
	accuracy $\epsilon_0  = f(x^0)-f(x^*)$ under the convergence rate $ \epsilon \leq
	(1-\sigma)^T \epsilon_0$, the number of needed steps can be bounded by
	\begin{equation}\label{eq:12}
	T \leq \frac{\log\left({\epsilon_t}\right) -\log\left({\epsilon_0}\right)  }{\log(1-\sigma)}. 
	\end{equation}
Using the bound derived for $1-\sigma(\hat S)$, we obtain \(T \leq
\big(\log(\frac{1}{\epsilon_t})-\log(\frac{1}{\epsilon_0})\big)
\big(\frac{\gamma^{p}}{\gamma^{d} + \gamma^{q}} + 1 \big)\).	Since,
the computation step is dominated by the inversion operation
$\mO(\mE[|\hat S|]^3)$, the number of arithmetic
operations is
	\[ \mO\left(\mE[|S|]^3\cdot T\right)\leq \mO\left(  \big(p +  R_d(p,q)\big)^3 \frac{\gamma^{p}}{\gamma^{d} + \gamma^{q}} \right).    \]
	The upper bound on the numerical effort in the previous
        equation has two regimes. At first, for small subset sizes it is
        increasing, but then exponential decay starts to dominate and
        using larger blocks significantly improves the convergence
        rate. Finally it flattens around $\mE[|\hat S|] = q$. Note that when $\lambda
        \approx \gamma^q$, i.e., for $q=\frac{\log(\lambda)}{\log(\gamma)}$, this
        phenomenon is visualized in Figure 
        \ref{fig:gamma} where the vertical bars correspond to $q$. In
        the regime where $d \approx q$, inverting the whole matrix
        seems to be the best option. When $q<d$, the term $\gamma^q$
        dominates the term $\gamma^d$, and the best subset size is
        either 1 or on the order of $q$, depending on the value of $\lambda$.
	
	These observations are contrary to the intuition from the
        previous works. We suspect that, due to fixed memory fetching
        costs, for small sizes the initial phase is unobserved but the
        second phase should be observed. Figure \ref{fig:gamma} 
        suggests that for sufficiently small values of $\lambda$ the
        numerical performance is maximized at the attenuation point
        $q$ and the predicted optimal block size is
          $\frac{\log(\lambda)}{\log(\gamma)}$.  
	
\subsection{Polynomially decaying spectrum}
	
Suppose that the eigenvalues $\{\lambda_i\}_{i=1}^d$
are decreasing polynomially, i.e., so that $\lambda_i
 = Ci^{-s} + C\lambda$ for $s > 1$. 
		
\paragraph{Motivation} For example, consider a
Mat\'ern kernel of order $s$, which has the form
$k(x,y)=C_2B_s(\norm{x-y})\exp(-C_1 \norm{x-y})$, where
$C_1$, $C_2$ are constants, and $B_s(d)$ is a modified
Bessel function of order $s$
\citep[see][]{Rasmussen2006a}. This class of kernels
exhibits asymptotically polynomial decay of eigenvalues
\citep[see][]{Seeger2008}.  

\paragraph{Complexity} Suppose, $\lambda = q^{-s}$ for
$q \in [1,\infty)$. To control the expected size let
us parameterize the tuning parameters as $\alpha
=Cp^{-s}$, where $C$ is a suitable general
constant. Then the convergence rate becomes:
\[1-\sigma(\hat{S}) = \frac{p^{-s}}{d^{-s}+p^{-s} + q^{-s} }= \frac{1}{(p/d)^{s} + (p/q)^{s}+ 1 } \]
and $ \mE[|\hat S|] = \sum_{i=1}^{d} \frac{ i^{-s} + q^{-s}}{i^{-s} + p^{-s}+q^{-s}}$. If $p \ll q$, we can establish by integral approximation that $\mE[|S|] \approx \mO(p)$, otherwise the expected size grows faster. Additionally, with increasing $p$ the convergence rate always improves. 
		
\paragraph{Numerical Effort} When $p \ll q$, similarly as in the preceding subsection,
the numerical cost becomes \( \mO\big(  p^3
  (\frac{d^sq^s}{p^s(q^s+d^s)}   )  \big)   \). This suggests that for $s \geq 3$ the
total numerical cost decreases for larger subsets, while for the problems with smaller $s$,
the cost increases. In general, it is difficult to obtain general
insights from the formulas, but the visalization in Figure
\ref{fig:exponential}b (right) suggests that if the regularization constant is
large (small $q$), even problems with large $s$ might incur more cost
as the subset size increases.  

This suggests that \emph{small block sizes matching the memory
  fetching costs} should be optimal if either the
regularization is large or if $s$ is small. With the same assumption,
if the desired accuracy is very high, performing \emph{full matrix
  inversion} can be more efficient, corresponding to $\mE[|\hat S|]
\rightarrow d$ in Figure~\ref{fig:exponential}b (right). Note that
increasing the accuracy to which we optimize the problem shifts the
curves up in the logarithmic plot, while keeping the end point fixed. 
\begin{figure*}
	\centering
	\begin{subfigure}[t]{0.49\textwidth}
		\includegraphics[width = 1\textwidth]{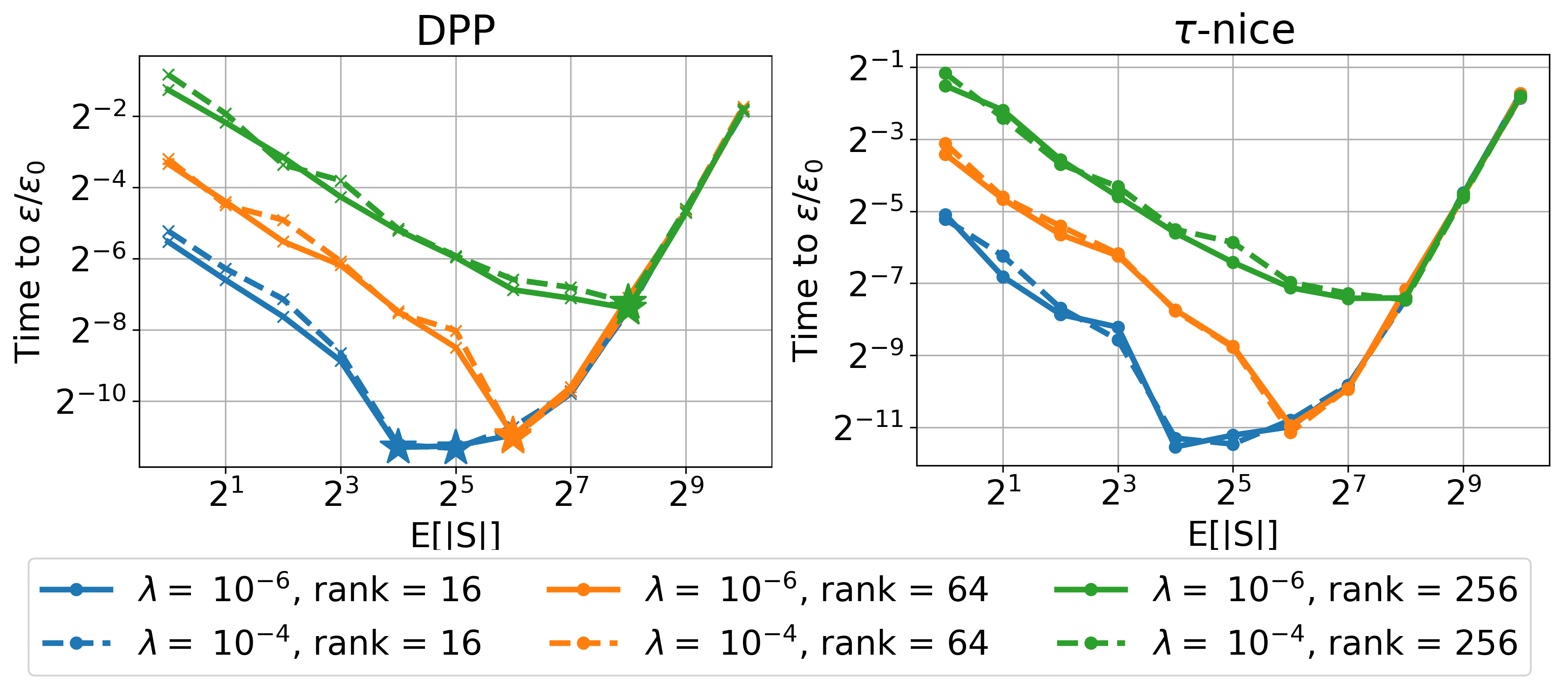}
                \vspace{-3mm}
		\caption{Sparse spectrum, rank of $\bK$ shown.}
		\label{fig:sparse}
	\end{subfigure}
	\begin{subfigure}[t]{0.49\textwidth}
\includegraphics[width = 1\textwidth]{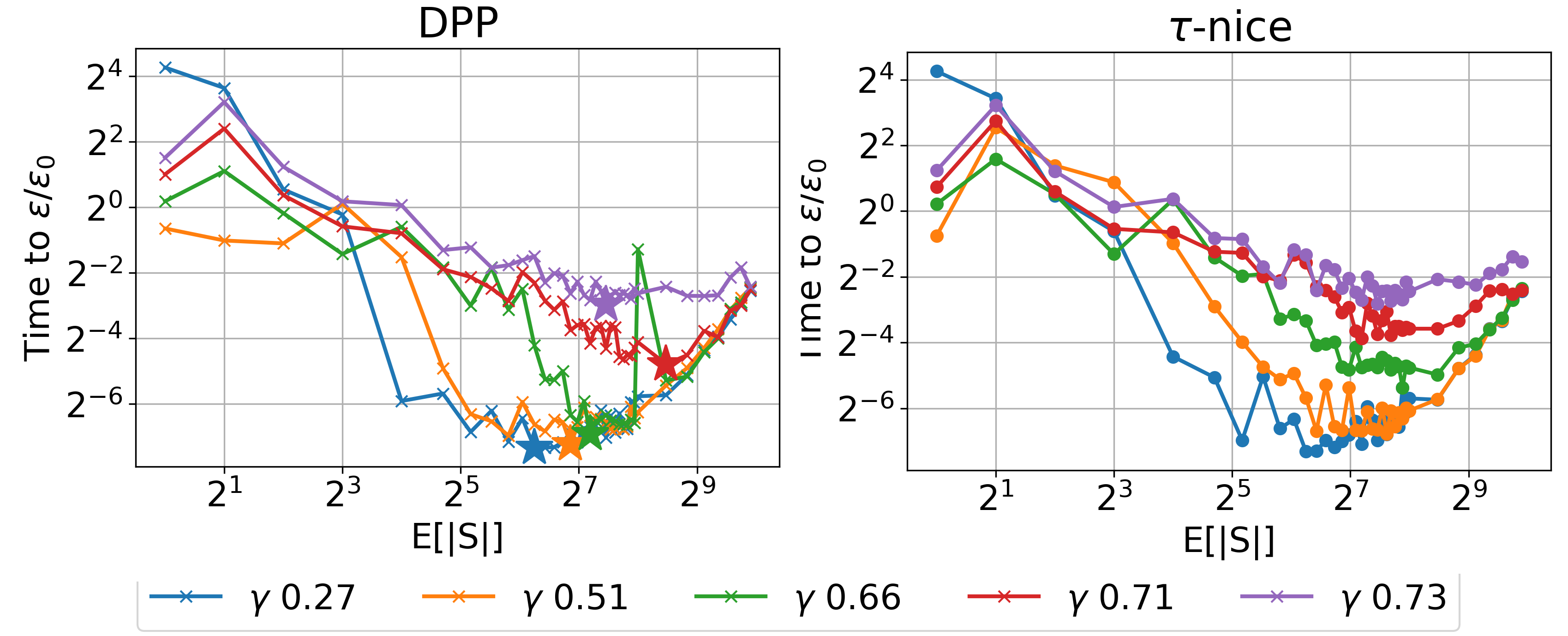}
\vspace{-3.5mm}
\caption{Exponential decay, varying lengthscale $\gamma$, for $\lambda = 10^{-7}$}
		\label{fig:gamma_exp}
	\end{subfigure}

	\caption{For Gaussian data, RNM exhibits
          similar behavior under DPP and uniform ($\tau$-nice) samplings.}
\end{figure*}

	\subsection{Sparse spectrum}
Suppose that only $s$ out of the $d$ eigenvalues are
        relatively large, while the remaining ones are very
        small. This scenario occurs with a linear kernel where 
        the number of large eigenvalues corresponds to the number of features,
        and the remaining ones are proportional to the
        regularization parameter $\lambda$.
	
	\paragraph{Complexity} For the ease of exposition, let the
        large eigenvalues all be equal to $\mu \gg \lambda$. 
        Lemma~\ref{eq:expected_size} implies that if $\alpha = \sum_{i> k-1}^{d}
        \lambda_i$ then $\mE[|\hat S|] \leq k$. 
	The convergence rate can be split to two cases:
	\begin{equation*}
	1- \sigma(\hat{S}) = \begin{cases}
	\frac{d-k}{d-k+1} & \text{when } k \in [s,d-1],\\
	1 - \frac{\lambda}{\mu}\frac{1}{d-k} + \mO\left(\frac{\lambda^2}{\mu^2}\right)& \text{when } k \in [0,s).
	\end{cases}
	\end{equation*}
	We see that once $k\geq s$ a discontinuity in the spectrum
        implies a discontinuity in the convergence rate. Consequently,
        the \emph{optimal subset size is of the order of $s$} as long as
        $\frac{\lambda}{\mu}$ is sufficiently small.

\section{EXPERIMENTS}
We numerically validate the theoretical findings from the previous sections.
Our main objective is to demonstrate that the convergence behavior of
RNM under DPP sampling aligns well with the behavior of RNM under
uniform sampling (called $\tau$-nice), which is more commonly used. This
would suggest that our convergence analysis under DPP sampling is also
predictive for other standard samplings. 
In addition to providing evidence for this claim, we also show that
there are cases where DPP sampling leads to superior performance of RNM.  

Even though there exist efficient algorithms for DPP sampling, we chose
to use approximate ridge leverage score sampling as a cheaper proxy for DPP
sampling, as suggested in a recent line of work
\citep{dpp-intermediate,dpp-sublinear}. The real data experiments were
performed with sampling according to the $\frac{1}{2}$-approximate
ridge leverage scores \citep{Calandriello2017}. We always report the mean value of $10$ reruns of the experiment with the given parameters. 

\paragraph{Gaussian Data}
The first experiment deals with data sampled from a Gaussian
distribution. The optimization using a kernel $\bK$ with sparse
spectrum (Figure \ref{fig:sparse}) verifies the theoretical findings
that the optimal block size should be of the same order as
$\operatorname{rank}(\bK)$. 
Using similarly generated data, and the relation in
\eqref{eq:se_decay} to relate lengthscale $l$ and $\gamma$ of squared
exponential kernel, we reproduce the prediction of the theory that for
sharper decays the optimal expected size should be larger (see Figure
\ref{fig:gamma_exp}, compared with theory, Figure \ref{fig:gamma}). The
performance of DPP and uniform sampling 
is on par as the intuition suggests, since for normally distributed
data even a uniform subsample provides good summary statistics. 
\paragraph{Gaussian Mixture Data}

\begin{figure}
	\centering
		\includegraphics[width = 0.49\textwidth]{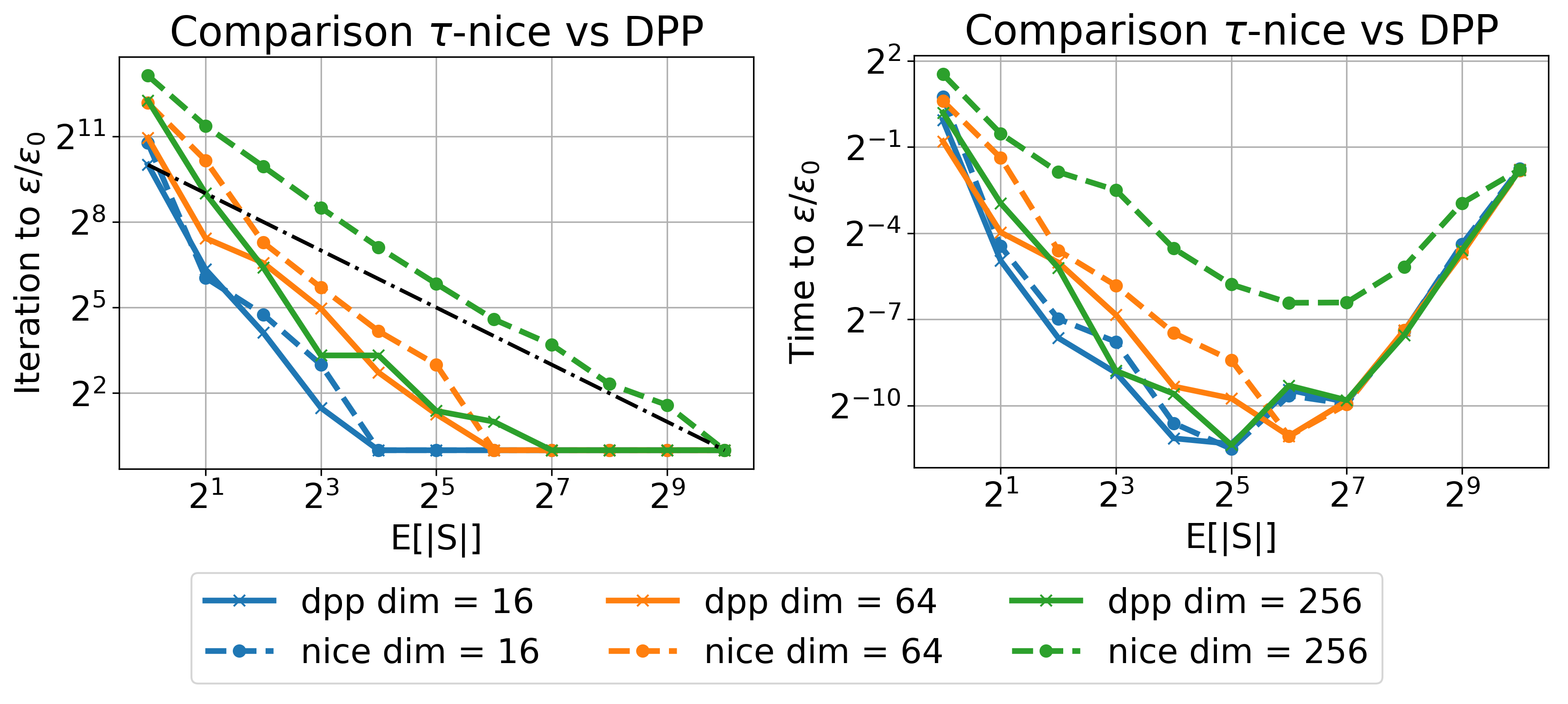}
		\caption{Gaussian mixture data - sparse spectrum, rank of kernel $\bK$ shown.}
		\label{fig:cluster}
\end{figure}
Akin to results from the sketching literature
\citep[e.g., see][]{unbiased-estimates-journal}, we suspect that the 
superior convergence of DPP sampling over uniform presents itself primarily
if the dataset is heterogeneous. By heterogeneity we mean that a
uniform subsampling of the points is likely not a 
good summary of the dataset. Consider a dataset where the points
are sampled from a Gaussian Mixture Model with $8$ clusters that are
equally likely. In order to have a good summary, a point from each
cluster should be present in the sample. DPP samples are generally
more diverse than uniform samples which makes it more likely that they
will cover all the clusters. In Figure \ref{fig:cluster}, we see that
DPP significantly outperforms uniform sampling for this dataset because it allows RNM 
to solve more representative subproblems. 

\paragraph{Real Data Experiments}
We perform two real data experiments on standard UCI datasets where we
optimize until statistical precision. In Figure~\ref{fig:real}a, we
optimize linear ridge regression on the \emph{superconductivity}
dataset. Next, in Figure \ref{fig:real}b we fit
kernel ridge regression with squared exponential kernel on the
\emph{cpusmall} dataset. For both datasets, the optimal subset size
under DPP sampling roughly matches the optimal size under uniform
sampling. Moreover, in the case of the superconductivity dataset, as
suggested by the theory for linear kernels, the optimal size is of
the same order as the feature dimensionality.
	\begin{figure}
		\centering
		\begin{subfigure}[t]{0.23\textwidth}
			\includegraphics[width = 1\textwidth]{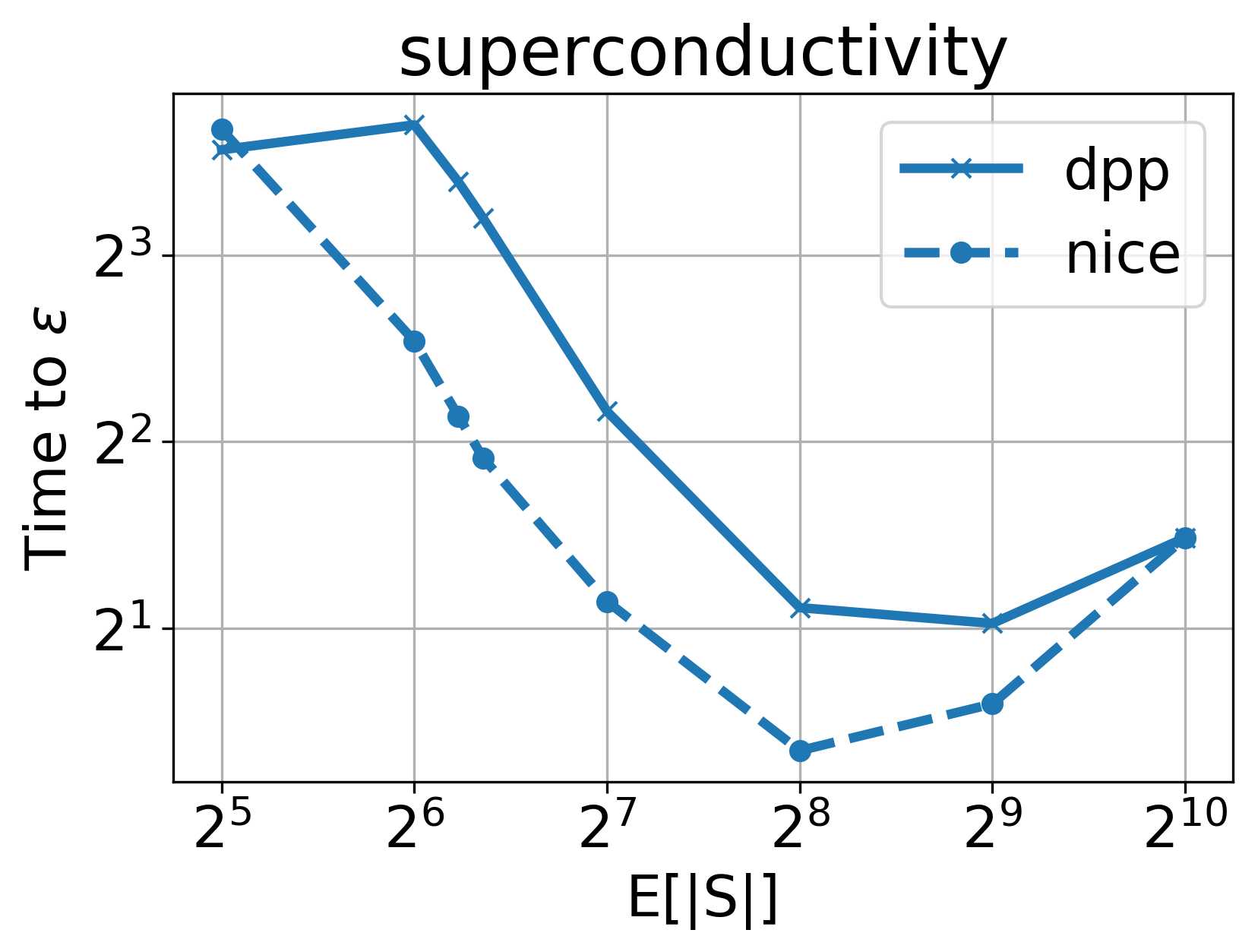}
			\caption{$n=1460$ and $d = 81$.}
			\label{fig:house}
		\end{subfigure}
		\begin{subfigure}[t]{0.23\textwidth}
			\includegraphics[width = 1\textwidth]{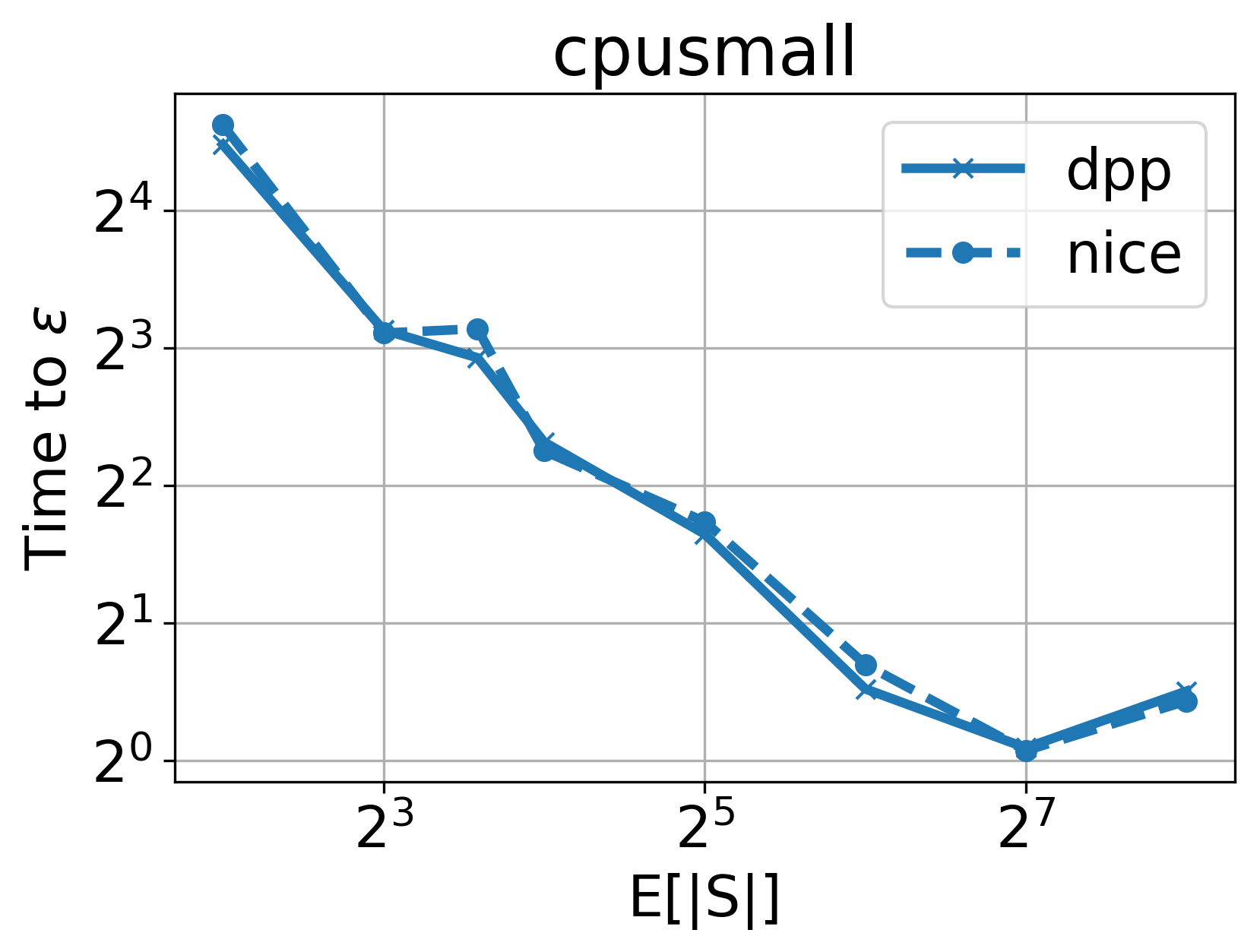}
			\caption{$n=2191$ and $d = 12$.}
			\label{fig:cpusmall}
                      \end{subfigure}
                      \vspace{2mm}
                      \caption{Experiments on real data.}
                      \vspace{1mm}
		\label{fig:real}
	\end{figure}

\section{CONCLUSION}
We analyzed a sampling strategy for the Randomized Newton Method, where
coordinate blocks of the Hessian over-approximation are sampled
according to their determinant. This sampling allows for a simple
interpretation of the convergence rate of the algorithm, which was
previously not well understood. We demonstrated that for empirical
risk minimization this convergence analysis allows us to predict the
optimal size for the sampled coordinate blocks in order to minimize
the total computational cost of the optimization. 

\paragraph{Acknowledgments} This work was supported by SNSF grant
407540\_167212 through the NRP 75 Big Data program. Also, MD thanks
the NSF for funding via the NSF TRIPODS program. 
\bibliography{pap.bib}
\newpage
\appendix
\onecolumn

\begin{center}
	\textbf{\large \underline{Supplementary Material:}\\ \underline{Convergence Analysis of Block Coordinate Algorithms}\\
			 \underline{with Determinantal Sampling}}
\end{center}

\section{PROOFS}

\subsection{DPPs}

\begin{proof}[Proof of Theorem \ref{t:main}]
	First, assume that $\alpha=1$. Since $\bM\succ \mathbf{0}$, we have $\det(\bM_{SS})>0$ for
	all $S\subseteq[d]$. We will next use the following
	standard determinantal formula which holds for any $v\in\mR^d$ and any
	invertible matrix $\bM$:
	\begin{align}
	\det(\bM)v^\top\bM^{-1}v = \det(\bM+vv^\top)
	- \det(\bM).\label{eq:pr0}
	\end{align}
	Applying this formula to the submatrices of $\bM$ and denoting by $v_S$
	the sub-vector of $v$ indexed by $S$, we show that for any $v\in\mR^d$:
	\begin{align*}
	v^\top
	&\mE\big[(\bM_S)^+
	\big]v =
	\sum_{S\subseteq[d]}\frac{\det(\bM_{SS})}{\det(\bI+\bM)}v_S^\top\bM_{SS}^{-1}v_S\\
	{\scriptsize\eqref{eq:pr0}}\ 
	&=\sum_{S\subseteq[d]}\frac{\det(\bM_{SS}+v_Sv_S^\top)-\det(\bM_{SS})}{\det(\bI+\bM)}\\
	&=\frac{\sum_S \det([\bM+vv^\top]_{SS}) - \sum_S\det(\bM_{SS})}{\det(\bI+\bM)}\\
	\text{\scriptsize(Lemma~\ref{l:normalization})}\
	&=\frac{\det(\bI+\bM+vv^\top) - \det(\bI+\bM)}{\det(\bI+\bM)}\\ 
	{\scriptsize\eqref{eq:pr0}}\
	&=\frac{\det(\bI+\bM)\,v^\top(\bI+\bM)^{-1}v}{\det(\bI+\bM)} \\
	&=
	v^\top(\bI+\bM)^{-1}v. 
	\end{align*}
	Since the above holds for all $v$, the equality also holds for the pd.
	matrices. To obtain the result with $\alpha\neq 1$, it suffices to
	replace $\bM$ with $\frac1\alpha\bM$.
\end{proof}

\begin{proof}[Proof of Lemma \ref{lemma:lambdas_exp}]
	The eigenvalues of $\bM(\alpha\bI+\bM)^{-1}$ are
	$\frac{\lambda_i}{\lambda_i+\alpha}$ so
	\begin{align*}
	\mE[|S|] & =  \sum_{i=1}^{d} \frac{\lambda_i}{\lambda_i + \alpha } =
	\sum_{i=1}^{d} \frac{\lambda_i}{\lambda_i + \sum_{j\geq k} \lambda_j } \\
	& = \sum_{i < k}^{d} \frac{\lambda_i}{\lambda_i +
		\sum_{j\geq k} \lambda_j } + \sum_{i \geq k}^{d}
	\frac{\lambda_i}{\lambda_i + \sum_{j\geq k} \lambda_j }\\
	&<  (k-1)+ 1 = k,
	\end{align*}
	which concludes the proof.
\end{proof}

\subsection{Convergence Analysis}
\begin{proof}[Proof of Theorem \ref{corr:recurence}]
	\begin{eqnarray}
	\sigma_1 & \stackrel{\eqref{eq:sigma_1}}= & \lambda_{\min}\left( \bM^{1/2} \left({\alpha}\bI +  \bM\right)^{-1} \bM^{1/2}  \right) \\
	& = & \lambda_{min} \left( \left( {\alpha}\bM^{-1} + \bI \right)^{-1} \right)\\
	& = & \frac{1}{\lambda_{\max}\left( 
		{\alpha} \bM^{-1} + \bI  \right)} = \frac{1}{1 + {\alpha}\lambda_{\max}(\bM^{-1})} \\
	& = & \frac{ \mu }{ \mu + \alpha} \\
	\end{eqnarray}
	where $\mu = \lambda_{\min}(\bM)$.
\end{proof}

\begin{proof}[Proof of Proposition \ref{prop:recurence}]
	By definition,
\[	\frac{1}{\sigma(k+1)} = 1 + \frac{\sum_{i>k}^{d}\lambda_i}{\lambda_d} = 1 + \frac{\sum_{i>{k-1}}^{d}\lambda_i - \lambda_{k}}{\lambda_d} = \frac{1}{\sigma(k)} - \frac{\lambda_{k}}{\lambda_d} \]
Rearranging,
\[ \frac{1}{\sigma(k)} = \frac{1}{\sigma(k+1)} +  \frac{\lambda_{k}}{\lambda_d}  \implies \sigma(k) = \frac{\sigma(k+1)\lambda_d}{\lambda_d + \lambda_k \sigma(k+1)}  \]
Dividing the denominator and the numerator by $\lambda_d$ finishes the proof. 

\end{proof}
%
%
%

\subsection{Dual convergence rate}
The dual convergence rate established in \cite{Qu2015Feb} relies on the
notion of expected separable over-approximation. Namely, the existence of
$v \in \mR^d$ s.t. $\mE[\bM_S] \preceq \bD(p \circ v)$, where $p$ is
the vector of marginal probabilities. In case of DPP sampling, one can
choose $v = \diag(\bM)\circ \diag(\bM(\bM + \alpha \bI)^{-1})^{-1}$,
and apply dual convergence results established in this literature. By
$\circ$ we denote element-wise product.

%
%
%
%
%
%
%
%
%
%

\section{LEVERAGE SCORE SAMPLING VS DPP SAMPLING}
\label{a:leverage}
We perform a simple experiment on the  Gaussian Mixtures dataset where the matrix has a sparse spectrum. In Figure \ref{fig:comparison} we see that the optimization process is influenced minimally. 

\begin{figure}[H]
	\centering
	\includegraphics[width = 0.5\textwidth]{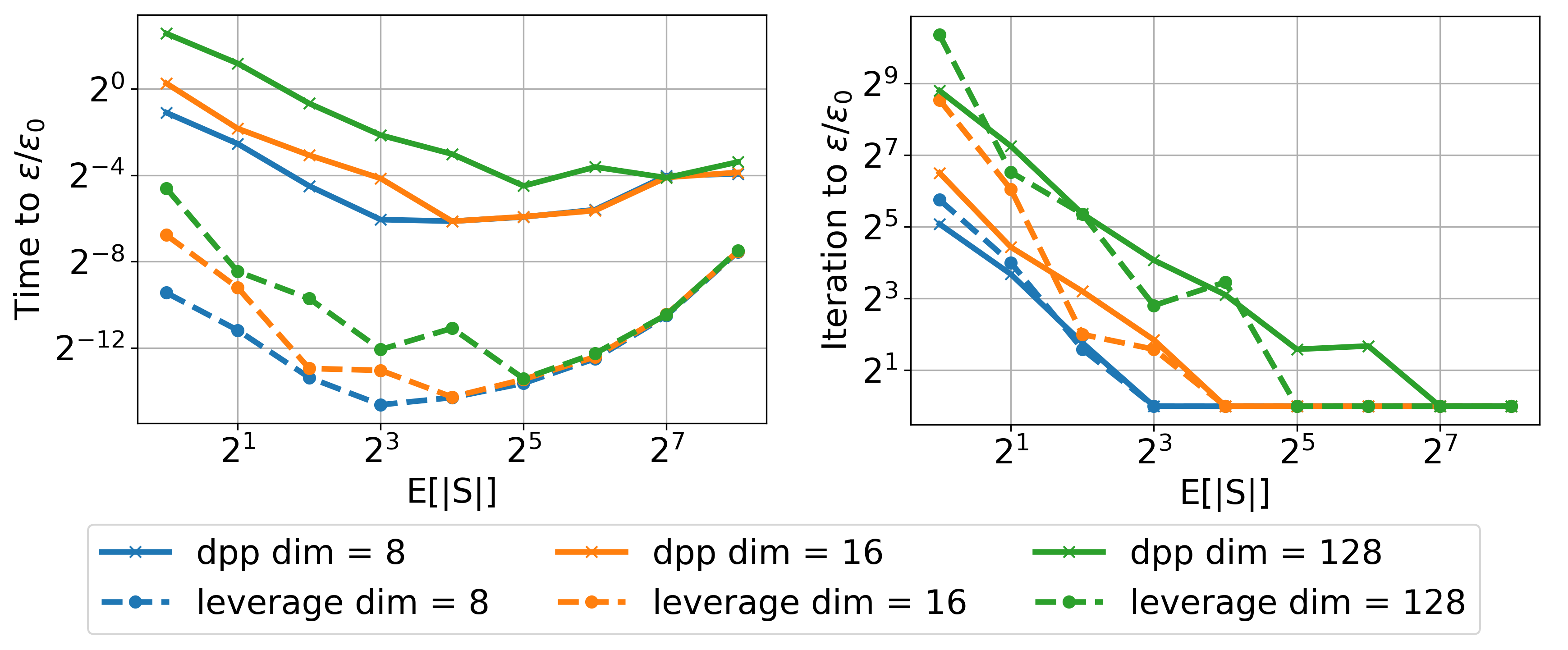}
	\caption{Comparison of leverage score sampling and DPP}
	\label{fig:comparison}
\end{figure}

\section{RELATIVE SMOOTHNESS AND RELATIVE STRONG CONVEXITY}\label{appendix:relative}
Recent works such as \citep{Gower2019} and \citep{Karimireddy2018a} introduce the concepts of relative-smoothness, relative strong convexity and $c$-stability. These are weaker conditions than assumed in this paper. With these conditions, the proof techniques used to analyze coordinate descent algorithms are applicable to Newton-like algorithms, where instead of a fixed matrix $\bM$, the actual Hessian $\bH(x)$ can be used. The extension to $c-$stability is done trivially in Theorem 2 of \cite{Karimireddy2018a}, here we focus on a slightly more elaborate connection with relative smoothness and relative strong-convexity. 

\begin{assumption}[\cite{Gower2019}]\label{def:assumption}
	There exists a constant $\tilde{L}\geq\tilde{\mu}$ such that for all $x,y \in \mQ \subseteq \mR^d$, where $\mQ:=\{x \in \mR^d : f(x)\leq f(x_0)\}$:
	\begin{equation}\label{eq:rsmooth}
			f(x) \leq f(y) + \braket{\nabla f(y),x} + \frac{\tilde{L}}{2}\norm{x-y}_{\bH(y)}
	\end{equation}
	and 
	\begin{equation}\label{eq:rconvex}
	f(x) \geq f(y) + \braket{\nabla f(y),x} + \frac{\tilde{\mu}}{2}\norm{x-y}_{\bH(y)}.
	\end{equation}
\end{assumption}

Now the task is to analyze the algorithm with the following update rule, which is identical to general Newton rule when $S = [d]$,

\begin{equation}\label{eq:update2}
x_{k+1} = x_k - \gamma (\bH(x_k)_{S_k})^+ \nabla f(x_k).
\end{equation}

We fix a particular choice of $\gamma = 1/\tilde{L}$. This should be contrasted with the update rule $\eqref{eq:update}$. 

Now given these assumption, we are able to show that the constant akin to $\sigma(\hat{S})$ appears in the analysis of this algorithm by utilizing the notions from \citep{Gower2019}. We sacrifice generality for the sake of brevity, and assume that range of $\bH(x)$ spans whole $\mR^d$ for each $x \in \mQ$. Then, the following quantities resembling $\sigma(\hat{S})$ appear in the convergence analysis of the update rule \eqref{eq:update2}

\begin{equation}\label{eq:sigmahat}
 \hat{\sigma}(\hat{S},x) = \lambda_{\min} \left( \mE_{\hat{S}} \left[ \bH^{1/2}(x)(\bH(x)_{\hat{S}})^+ \bH(x)^{1/2}  \right] \right)
\end{equation}

and \[\hat{\sigma}(\hat{S}) = \min_{x\in \mQ } \hat{\sigma}(\hat{S},x) \]

\begin{theorem}[Theorem 3.1 of \cite{Gower2019}, modified]\label{thm:rel}
 	Let $f$ satisfy Assumption \ref{def:assumption}, and let $\bH(x)$ be the Hessian at $x$ having range that spans whole $\mR^d$ for all $x$. Then 
	\[ \mE_{\hat{S}}[ f(x_{k+1}) - f(x^*) ] \leq  \left( 1 - \frac{\hat{\sigma}(\hat{S},x_k)\mu}{L}  \right) (f(x_k) - f(x^*), \]
	and
	\[ \mE_{\hat{S}}[ f(x_{k}) - f(x^*) ] \leq  \left( 1 - \frac{\hat{\sigma}(\hat{S})\mu}{L}  \right)^k (f(x_0) - f(x^*), \]
	where $ \hat{\sigma}(\hat{S}) = \min_{x\in \mQ } \hat{\sigma}(\hat{S},x)$ as in Equation \eqref{eq:sigmahat}.
\end{theorem}

\begin{proof}
	Minimizing the upper bound in \eqref{eq:rsmooth} restricted to coordinates in $S_k$, we arrive at, 
	\begin{eqnarray*}
	f(x_{k+1}) - f(x_k) & \stackrel{\eqref{eq:smooth}} \leq & - \frac{1}{2\tilde{L}} \braket{\nabla f(x_k), (\bH(x_k)_{S_k})^+  \nabla f(x_k)} \\
	\mE[f(x_{k+1}) - f(x_k)] & \leq &  - \frac{1}{2\tilde{L}} \braket{\nabla f(x_k),\mE_{\hat{S}}[(\bH(x_k)_{S_k})^+]  \nabla f(x_k)}\\
	& \stackrel{\eqref{eq:rconvex}, \eqref{eq:sigmahat}} \leq & -\frac{\mu}{L}\hat{\sigma}(\hat{S},x) (f(x_k) - f(x^*))\\
	& \leq & -\frac{\mu}{L}\hat{\sigma}(\hat{S}) (f(x_k) - f(x^*))\\
	\end{eqnarray*}
	rearranging finishes the proof. 
\end{proof}

The following corollary states that with DPP sampling, the update rule in \eqref{eq:update2} can have a more interpretable convergence rate than stated in the Theorem \ref{thm:rel}. 

\begin{corollary}[of Theorem \ref{thm:rel}]
Under the assumption of Theorem \ref{thm:rel}, let additionally $S_k$ be a sample from sampling $\hat{S}_k \sim \DPP(\frac{1}{\alpha} \bH(x_k))$, then 
		\[ \mE_{\hat{S}_k}[ f(x_{k+1}) - f(x^*) ] \leq  \left( 1 - \left(\frac{\lambda(x_k) }{\lambda(x_k)  + \alpha}\right)\frac{\mu}{L}  \right) (f(x_k) - f(x^*), \]
		where $\lambda(x_k) = \lambda_{\min}(\bH(x_k))$. 
\end{corollary}

The following lemma relates the complexity quantity defined above to the definition of $\sigma(\hat{S})$ used in the main body of this paper. Note that $\hat{\sigma}$ is larger than $\sigma$, even if the fixed over-approximation exists, as previously we assumed the over-approximation to be valid globally not just in  $\mQ$.

\begin{lemma} If for all $x \in \mQ$, $\bM \succeq \bH(x) \succeq \kappa \bM \succ 0$, then
	\[ \hat{\sigma}(\hat{S}) \geq \kappa \sigma(\hat{S}).\]
	The relative smoothness, and strong-convexity can be chosen to be $\tilde{L} = 1$, and $\tilde{\mu} = 1$, respectively. 
\end{lemma}
\begin{proof}
\begin{eqnarray*}
	\hat{\sigma}(\hat{S}) & = & \min_{x\in \mQ } \min_{v \in \mR^d} \frac{ \braket{v,\mE_{\hat{S}} \left[ \bH^{1/2}(x)(\bH(x)_{\hat{S}})^+ \bH(x)^{1/2}  \right]v }}{{\norm{v}_2^2}} = \min_{v \in \mR^d} \min_{x\in \mQ}   \frac{ \braket{v,\mE_{\hat{S}} \left[ \bH^{1/2}(x)(\bH(x)_{\hat{S}})^+ \bH(x)^{1/2}  \right]v }}{\norm{v}_2^2} \\
	& \geq & \min_{v \in \mR^d}  \frac{ \braket{v,\mE_{\hat{S}} \kappa \left[ \bM^{1/2}(\bM_{\hat{S}})^+ \bM^{1/2}  \right]v }}{\norm{v}_2^2} = \kappa \sigma(\hat{S})
\end{eqnarray*}

\end{proof}

\section{OTHER SAMPLINGS}

The convergence properties of RNM with determinantal sampling depend solely on the spectral properties of $\bM$. This is not true of other common samplings such as $\tau$-nice. Indeed we can improve or worsen the performance of $\tau$-nice sampling when $\bM$ is transformed via spectrum preserving operation such as unitary transformation \[\bM \gets \bR^\top \bM \bR, \text{ where } \bR^\top \bR = \bI.\]

Suppose that we are given an eigenvalues of the matrix $\bM$, for any sampling $\hat{S}$ is it possible to find a spectrum preserving rotation such that $\sigma(\hat{S})$ is at least as small as $\sigma(\hat{S}_{\DPP})$ which corresponds to DPP sampling with the same expected cardinality? The answer turns out to be negative, and we show counter-example. 

\begin{remark}[Counter-example]
	Let $\hat{S}_1$ be a sampling such that $[n]$ is sampled with $1/2$ probability and $\emptyset$ and $1/2$ probability. The expected size of the subset $\mE[|\hat{S}_1|] = d/2$ and $\sigma(\hat{S}_1) = \frac{1}{2}$ irrespective of the matrix $\bM$.
	
	Suppose matrix $\bM$ has degenerate spectrum such that $\lambda$ is eigenvalue with multiplicity $d/2$ and $\mu$ is eigenvalue with $d/2$ multiplicity where $\lambda < \mu$. In order s.t. $\mE[|S_{\DPP}|] = \frac{d}{2}$, $\alpha = \sqrt{\lambda \mu}$, then $\sigma(\hat{S}_{\DPP}) < \frac{1}{2}$. 
\end{remark}
In what circumstances does DPP sampling perform better than a uniform sampling? First, we consider circumstances where uniform sampling is optimal. 

\subsection{Uniform sampling}
It is important to allow for variation in the off-diagonal of $\bM$. If we consider only diagonal $\bM$, the optimal sampling is uniform sampling. 
\begin{lemma} \label{lemma:uniform_optimal} Let $\bM$ be diagonal. The quantity $\sigma(\hat{S})$ of a sampling over a power set $P([d])$ constrained by $\mE[|\hat{S}|] = k$ is maximized for uniform samplings.
\end{lemma}
\begin{proof}[Proof of Lemma \ref{lemma:uniform_optimal} ]
	We want to maximize the minimum eigenvalue of a matrix $\bM^{1/2} \mE[(\bM_S)^{-1}  ] \bM^{1/2}$. For a diagonal $\bM$ we know that $(\bM_S)^{-1} = (\bM^{-1})_S$. Hence, \(\bM^{1/2} \mE[(\bM_S)^{-1}  ] \bM^{1/2} \bD(p)\), where $p$ is a vector of marginals $p_i = P(i \in \hat{S})$. Hence, the minimum eigenvalue is the minimum marginal probability subject to a constraint that $\mE[|S|] = \sum_{j=1}^{d}P(j \in \hat{S}) \leq k$. This leads to an optimum where $P(i \in \hat{S}) = P(j \in \hat{S})$ for all $i,j \in [d]$. Hence the optimal sampling distribution is uniform. 
\end{proof}

\subsection{Parallel Sampling}
The parallel extension of the update method \ref{eq:update} has been considered in \cite{Mutny2018a} and \cite{Karimireddy2018a}. Namely, the authors consider a case, when the updates with $c$ machines are aggregated together to form a single update in the form $\approx \frac{1}{b}\sum_{j=1}^{c} (\bM_{S_j})^+$, where $b$ is the aggregating parameter. It is known that for parallel disjoint samplings the convergence rate increases linearly with the number of processors. For independent samplings the aggregating parameter $b$ depends on the quantity, 
\[ \theta(\hat{S}) = \lambda_{\max} (\bM^{1/2}\mE[(\bM_{\hat{S}})^+] \bM^{1/2} ) \]
which in the case of DPP sampling is equal to $\theta = \frac{\lambda_1}{\lambda_1 + \alpha}$. The quantity $\theta (\hat{S}) \in [\sigma(\hat{S}),1]$, and as $\theta \rightarrow 1$, the aggregation operation becomes averaging $b \rightarrow c$. For DPP sampling, we can see an inverse relationship between increasing $\sigma(\hat{S})$ by increasing block size, which inherently makes the parallelization problem more difficult by increasing $\theta(\hat{S})$.

\end{document}